\documentclass[11pt]{amsart}
\usepackage{graphicx}
\usepackage{enumerate}
\usepackage{amsmath, amsfonts, amsthm}
\usepackage{amssymb}
\usepackage{shuffle}
\usepackage{color}
\usepackage{todonotes}
\presetkeys%
    {todonotes}%
    {inline,backgroundcolor=yellow}{}
\numberwithin{equation}{section}

\newcommand{\comm}[1]{}
\newtheorem{theorem}{Theorem}
\newtheorem{definition}[theorem]{Definition}
\newtheorem{lemma}[theorem]{Lemma}
\newtheorem{question}[theorem]{Question}
\newtheorem{remark}[theorem]{Remark}
\newtheorem{proposition}[theorem]{Proposition}

\numberwithin{theorem}{section}

\newtheorem{thm}[equation]{Theorem}
\newtheorem{cor}[equation]{Corollary}
\newtheorem{lem}[equation]{Lemma}

\theoremstyle{remark}

\DeclareMathOperator{\Isom}{Isom}

\DeclareMathOperator{\diam}{diam}

\newcommand{\N}{\mathbb{N}}

\newcommand{\Ga}{\Gamma}

\newcommand{\HH}{\mathbb{H}}

\newcommand{\Stab}{\textup{Stab}}

\newcommand{\CAT}{\text{CAT}}
\newcommand{\Cay}{\text{Cay}}

\title{Actions on $\CAT(-1)$ spaces with critical exponent less than $1$}
\author{Beibei Liu and Shi Wang}
\date{\today}
\address{Department of Mathematics, The Ohio State University, 100 Math Tower, 231 W 18th Ave, Columbus, OH 43201, USA}
\email{liu.11302@osu.edu}

\address{Institute of Mathematical Sciences,
	ShanghaiTech University, Pudong, Shanghai, China}
\email{wangshi@shanghaitech.edu.cn}

\dedicatory{To the memory of Misha Kapovich}

\begin{document}
\maketitle
\begin{abstract} We show that for a discrete isometry subgroup acting on a proper $\CAT(-1)$ space $X$, if the critical exponent is less than $1$, then the critical exponent equals the Hausdorff dimension of the entire limit set. Consequently, the limit set must be a Cantor set. As an application, we prove that any finitely generated, torsion-free discrete subgroup in $\Isom(X)$ with critical exponent less than one must be geometrically finite and free. This answers a question of Kapovich.
\end{abstract}

\section{Introduction}
Let $X$ be a proper $\mathrm{CAT}(-1)$ space, and let
$\Gamma<\operatorname{Isom}(X)$ be a discrete subgroup. Fix a base point $o\in X$ and denote $d$ the distance function on $X$. The Poincar\'e series of the action is
defined by
\[
P_\Gamma(s,o)
   :=\sum_{\gamma\in\Gamma} e^{-s d(o,\gamma o)},\quad s\in (0,\infty),
\]
whose abscissa of convergence is called the critical exponent of the $\Ga$-action, denoted by
\[
\delta(\Gamma)
   :=\inf\bigl\{
      s>0 : P_\Gamma(s,o)<\infty
   \bigr\}.
\]
The value of $\delta(\Gamma)$ is independent of the choice of $o$.
Equivalently, it measures the exponential growth rate of the $\Ga$-orbits in $X$. The main purpose of this paper is to prove the following theorem.

\begin{thm}\label{thm:main}
Let $X$ be a proper $\mathrm{CAT}(-1)$ space, and let
$\Gamma<\operatorname{Isom}(X)$ be a finitely generated, non-elementary discrete
group. Suppose for some $\alpha\in(0,1]$ the Poincar\'e series
\[
P_\Gamma(\alpha,o)
   =
   \sum_{\gamma\in\Gamma}
      e^{-\alpha d(o,\gamma o)}
   <\infty.
\]
Then the limit set has Hausdorff measure
\[
\mathcal H^\alpha\bigl(\Lambda(\Gamma)\bigr)=0.
\]
In particular, $\Lambda(\Ga)$ is homeomorphic to the Cantor set and its Hausdorff dimension $$\dim_{\mathcal H}(\Lambda(\Ga))= \delta(\Ga).$$
\end{thm}

\begin{remark}
The Hausdorff measure/dimension are associated with the visual metric $\rho_o$ on $\Lambda(\Ga)\subset \partial_\infty X$. (See Section \ref{sec:geometry}.) Changing the basepoint yields a bi-Lipschitz equivalent visual metric, hence the vanishing of
$\mathcal H^\alpha(\Lambda(\Gamma))$, and the Hausdorff
dimension of $\Lambda(\Gamma)$, are both independent of the choice of the
basepoint.
\end{remark}

As an application, using Bowditch's structural results of groups acting on Cantor sets \cite{BowditchCantor}, we obtain the following.

\begin{thm}\label{thm:gf}
Let $X$ be a proper $\mathrm{CAT}(-1)$ space, and let
$\Gamma<\operatorname{Isom}(X)$ be a finitely generated, torsion-free, discrete group with $\delta(\Ga)<1$, then $\Ga$ is geometrically finite (in the sense of Definition \ref{def:GF}) and free.
\end{thm}

\begin{remark}
    Our proof relies on the torsion-freeness of $\Ga$. Note that for general $\CAT(-1)$ spaces, Selberg's lemma fails \cite{KapovichSelberg}.
\end{remark}

In particular, Theorem \ref{thm:gf} answers a question \cite[Problem 1.6]{Ka} of Kapovich.

\begin{cor}[Kapovich's conjecture]\label{cor:Kapovich}
Let $\Gamma<\Isom(\HH^n)$ be a finitely generated discrete subgroup satisfying $ \delta(\Gamma)<1. $ Then $\Gamma$ is geometrically finite and virtually free. 
\end{cor}

\begin{proof} By Selberg's lemma, there exists a torsion-free finite-index subgroup $\Ga_0<\Ga$, and $\delta(\Ga_0)=\delta(\Ga)<1$. By Theorem \ref{thm:gf}, $\Ga_0$ is geometrically finite and free. Hence $\Ga$ is geometrically finite and free. 
\end{proof}

\begin{remark}
    Corollary \ref{cor:Kapovich} also follows directly from Theorem \ref{thm:main}, using the structural results of Kulkarni \cite{Kulkarni} and Gusevskii \cite{Gus89}. See also \cite{PankkaSouto} and \cite[Corollary 7.11]{Kapovich}.
\end{remark}
 
\begin{remark} In \cite[Problem 1.6]{Ka} and \cite[Conjecture 10.14]{Kapovich}, Kapovich asked further whether every finitely generated Kleinian group with critical exponent less than one must be of classical Schottky type. This is stronger than the virtual-freeness and geometric-finiteness conclusions proved above. We do not pursue the question here, and we refer to \cite{Hou10, Hou20, Hou23} for related results. 
 \end{remark}

\begin{remark}
    Kapovich also conjectured \cite[Conjecture 10.17]{Kapovich} that under a milder condition $\delta(\Gamma)<2$, it is sufficient to imply the geometric finiteness of $\Ga$. This is true for $n=3$, which follows from the tameness Theorem, the Ahlfors measure conjecture, and the work of Bishop–Jones (See also \cite{BishopJones,Ka}). However, the current method fails in the general range of $1\leq\delta(\Ga)<2$. In fact, the condition $\delta(\Ga)<1$ closely lines up with the $-1$ curvature bound.
\end{remark}

We mention a few more prior works relating to the question.
\begin{enumerate}
    \item[$\bullet$] Pankka–Souto \cite{PankkaSouto} proved that if a Kleinian group (not necessarily finitely generated) has Hausdorff dimension less than $1$, then it is free.
    \item[$\bullet$] Chang–Qing–Yang \cite{CQY} proved that a conformally finite Kleinian group is geometrically finite unless its limit set has the maximal possible Hausdorff dimension.   
    \item[$\bullet$] When $X$ is a pinched negatively curved Riemannian manifold, the authors \cite{LiuWang} proved that if $\Ga$ has sufficiently small $\delta(\Ga)$ (depending on the dimension and pinching constant), then $\Ga$ is convex cocompact.
    \item[$\bullet$] Kapovich showed \cite{Kapovich} that when $X=\HH^n$, the virtual homological dimension of $\Ga$ is bounded by $\delta(\Ga)+1$. As a result, if $\Ga$ is finitely presented with $\delta(\Ga)<1$, it must be virtually free. (See also \cite{CFM, CMW25, CMW26} for further generalizations.)
\end{enumerate}

\subsection*{The structure of limit sets} The visual boundary $\partial_\infty X$ is the disjoint union of the limit set $\Lambda(\Ga)$ and the domain of discontinuity $\Omega(\Ga)$, while the limit set itself decomposes into the sets of conical and non-conical limit points. (See Section \ref{sec:limit} for precise definitions of $\Lambda_c(\Ga)$ and $\Lambda_{nc}(\Ga)$.) These decompositions are important both geometrically and dynamically, as they encode the recurrence properties of geodesic rays in the quotient space $\Ga\backslash X$, relative either to the convex core or to arbitrary compact subsets.

By the Bishop--Jones theorem and its generalizations \cite{BishopJones,Paulin97,DSU17,Cavallucci}, the critical exponent of a non-elementary discrete group of isometries of a proper $\CAT(-1)$ space satisfies
\[\delta(\Ga)=\dim_{\mathcal H}\bigl(\Lambda_c(\Ga)\bigr).\]
Consequently, $\delta(\Ga)$ may be strictly smaller than the Hausdorff dimension of the entire limit set. Thus, any gap between the critical exponent and the Hausdorff dimension of the full limit set can arise only from the size of the non-conical limit set. The Hausdorff dimension of nonconical limit sets has been studied in
a variety of settings. See
\cite{FernandezMelian,BishopJonesLIL,Gonye,
KapovichLiuNonconical,MjYang}.

The structure and size of the conical and non-conical parts of the limit set are closely related to the notion of geometric finiteness (See Definition \ref{def:GF}). Geometrically finite groups form a fundamental class of negatively curved actions with well-behaved geometric and dynamical properties \cite{BeardonMaskit, Patterson76,Sullivan79, Bow93, Bow95, Roblin}. For such groups, the set of non-conical limit points is countable, and 
\[\delta(\Ga)=\dim_{\mathcal H}\bigl(\Lambda(\Ga)\bigr).\]

On the contrary, for geometrically infinite groups, the full limit set can a priori be much larger than what is detected by the critical exponent. Indeed, for every $\epsilon>0$, Patterson \cite{Patterson83} constructed a discrete, non-elementary, infinitely generated group $\Gamma_\epsilon
   <\operatorname{Isom}(\mathbb H^n)$
such that $\Lambda(\Ga)=S^{n-1}$ and $0<\delta(\Gamma_\epsilon)\leq\epsilon.$ There is a more systematic way of producing such examples by constructing non-amenable covers of compact locally $\CAT(-1)$-manifolds.  For example, using an adapted Rips’ construction of Wise \cite{Wise}, one can create a finitely generated normal subgroup of a $\CAT(-1)$ group whose quotient is the free group $F_2$. In particular, these examples satisfy \cite{CDS18}
\[\delta(\Ga)<\dim_{\mathcal H}\bigl(\Lambda(\Ga)\bigr).\]

Our Theorem \ref{thm:main} and $\ref{thm:gf}$ simply show that under the condition $\delta(\Ga)<1$, the structure of $\Lambda(\Ga)$ must be simple enough. We do not know if the condition $\delta(\Ga)<1$ is sharp. There are geometrically infinite surface Kleinian groups constructed by \cite{Jorgensen77} whose critical exponent is $2$, so the sharp constant should lie in $[1,2]$. Motivated by this, we ask the following question.

\begin{question}
    What is the smallest real number $\delta_0\geq 1$ such that the following statement holds? If $\Gamma$ is a finitely generated, torsion-free, discrete group acting on a proper $\CAT(-1)$ space whose critical exponent is $<\delta_0$, then $\Ga$ must be geometrically finite.
\end{question}

\subsection*{AI declaration} The initial idea of using a spanning tree together with the Poincar\'e series to control the $\alpha$-energy (Section \ref{sec:AI}) was suggested by OpenAI's ChatGPT (GPT-5.6 Sol). The authors subsequently verified and developed the argument in full generality and take complete responsibility for all mathematical claims and other contents of this article.

\subsection*{Acknowledgement}
SW was partially supported by NSFC-12301085.

\section{\CAT(-1) geometry}\label{sec:geometry}

\subsection{Visual metric.}
We let $X$ denote a proper $\CAT(-1)$ space. The visual boundary $\partial_\infty X$ consists of asymptotic classes of geodesic rays, where two rays $c_1,c_2:[0,\infty)\to X$ are asymptotic if 
\[\sup_{t\geq 0}d(c_1(t),c_2(t))<\infty.\]
Equipped with the cone topology, the union $\overline{X}=X\cup \partial_\infty X$ gives a natural compactification of $X$, known as the visual compactification.

Using the Gromov product, we can describe the cone topology concretely. Recall that for any $x,y\in X$, the Gromov product (based at $o$) is given by
\[(x,y)_o=\frac{1}{2}
   \Bigl(
      d(o,x)
      +d(o,y)
      -d(x,y)
   \Bigr).\]
The definition naturally extends to $\overline X$. If $x\in X$ and $\xi\in \partial_\infty X$, then
\[(x,\xi)_o
   :=
   \lim_{t\to\infty}
   \frac{1}{2}
   \Bigl(
      d(o,c_\xi(t))
      +d(o,x)
      -d(c_\xi(t),x)
   \Bigr),\]
and if $\xi,\eta\in \partial_\infty X$, then
\[
(\xi,\eta)_o
   :=
   \lim_{t\to\infty}
   \frac{1}{2}
   \Bigl(
      d(o,c_\xi(t))
      +d(o,c_\eta(t))
      -d(c_\xi(t),c_\eta(t))
   \Bigr).
\]
Here and throughout the paper, $c_\xi$ and $c_\eta$ denote the unit-speed geodesic rays emanating
from $o$ and terminating at $\xi$ and $\eta$, respectively. Then any sequence $x_n\in X$ converges to $\xi\in \partial_\infty X$ if and only if $$(x_n,\xi)_o\to \infty,$$ and any sequence $\xi_n\in \partial_\infty X$ converges to $\xi\in \partial_\infty X$ if and only if $$(\xi_n,\xi)_o\to \infty.$$

The visual boundary also carries a
canonical metric, called the visual metric or the Bourdon metric
\cite{Bourdon95,BridsonHaefliger}, defined by
\[
\rho_o(\xi,\eta)
   :=e^{-(\xi,\eta)_o},
   \qquad \xi,\eta\in\partial_\infty X,
\]
The $\mathrm{CAT}(-1)$ condition guarantees that $\rho_o$ is a genuine
metric, and properness shows $\partial_\infty X$ is compact. Thus, with the visual metric, $\partial_\infty X$ is a compact (hence complete) metric space with diameter $\leq 1$. For a different choice of basepoint $o'\in X$, we have by the triangle inequality that
\[
e^{-d(o,o')}\rho_o
   \leq
   \rho_{o'}
   \leq
   e^{d(o,o')}\rho_o.
\]
so the visual metrics $\rho_o$ and $\rho_{o'}$ are bi-Lipschitz equivalent. For a more detailed exposition of the visual boundary and the Bourdon metric, we refer the reader to
\cite[\S\S~1.4--1.5 and 2.4--2.6]{Bourdon95};
see also \cite[Chapter~III.H]{BridsonHaefliger}.

We will use the following properties of the Gromov product in Section \ref{sec:proof}.

\begin{lemma}
\label{lem:gromov-product-lipschitz}
For any $x,y\in X$ and $\xi\in\partial_\infty X$, we have
\[
\bigl|(x,\xi)_o-(y,\xi)_o\bigr|
\leq d(x,y).
\]
In particular,
\[
(x,\xi)_o
\geq
(y,\xi)_o-d(x,y).
\]
\end{lemma}

\begin{proof}
Let $c_\xi:[0,\infty)\to X$ be the geodesic ray from $o$ to $\xi$.
By definition,
\[
(x,\xi)_o
=
\lim_{t\to\infty}(x,c_\xi(t))_o,
\]
where
\[
(x,c_\xi(t))_o
=
\frac12
\left(
d(o,x)+d(o,c_\xi(t))-d(x,c_\xi(t))
\right).
\]
For every $t\geq0$, the triangle inequality gives
\[
\begin{aligned}
\bigl|
(x,c_\xi(t))_o-(y,c_\xi(t))_o
\bigr|
&\leq
\frac12
\left(
|d(o,x)-d(o,y)|
+
|d(x,c_\xi(t))-d(y,c_\xi(t))|
\right)\\
&\leq d(x,y).
\end{aligned}
\]
Letting $t\to\infty$ yields
\[
\bigl|(x,\xi)_o-(y,\xi)_o\bigr|
\leq d(x,y),
\]
as desired.
\end{proof}

\begin{lemma}\label{lem:gromov-prod}
Let $\xi, \eta\in \partial_\infty X$. If $p\in c_\xi$ and $q\in c_\eta$, then
\begin{equation*}
  (\xi,\eta)_o \geq (p,\eta)_o
  \geq (p,q)_o.
\end{equation*}
\end{lemma}

\begin{proof}
We prove the first inequality. Let $t>0$ be large enough so that $p\in [o,c_\xi(t)]$. Using the triangle inequality, we obtain that 
\begin{align*}
d\bigl(c_\xi(t),c_\eta(t)\bigr)&\leq d(c_\xi(t),p)+d(p,c_\eta(t))\\
&= \bigl(t-d(o,p)\bigr)+d(p,c_\eta(t)).
\end{align*}
It follows that
\[\frac{1}{2}\left[2t-d\bigl(c_\xi(t),c_\eta(t)\bigr)\right]\geq \frac{t+d(o,p)-d(p,c_\eta(t))}2\]
for any large $t$. By taking the limit $t\to \infty$, we obtain 
$(\xi,\eta)_o\geq (p,\eta)_o$.
The second inequality follows similarly.
\end{proof}

\subsection{Limit sets and geometric finiteness}\label{sec:limit}
Let $\Ga<\Isom(X)$ be a discrete isometry subgroup acting on a proper $\CAT(-1)$-space $X$, and $o\in X$ be a chosen basepoint. The limit set $\Lambda(\Gamma)\subset \partial_\infty X$ consists of the accumulation points of $\{\Gamma o\}$. The group $\Gamma$ is said to be non-elementary if $|\Lambda(\Gamma)|\geq 3$. Note that the definition of $\Lambda(\Ga)$ is independent of the choice of $o\in X$. 

We call a limit point $\xi\in\Lambda(\Gamma)$ \emph{conical}, or
\emph{radial}, if there exist a constant $R>0$, a sequence
$t_j\to\infty$, and elements $\gamma_j\in\Gamma$ such that
\[
d\bigl(c_\xi(t_j),\gamma_j o\bigr)\leq R
\qquad\text{for every }j.
\]
We denote the set of conical limit points by $\Lambda_c(\Gamma)$ and
its complement in the limit set by
\[
\Lambda_{nc}(\Gamma)
   :=\Lambda(\Gamma)\setminus\Lambda_c(\Gamma),
\]
the \emph{nonconical limit set}.

For $\xi\in\Lambda(\Gamma)$, let 
\[ \Gamma_\xi := \Stab_\Gamma(\xi) = \{\gamma\in\Gamma:\gamma\xi=\xi\}. \] 
A point $\xi\in\Lambda(\Gamma)$ is called a \emph{parabolic point} if $\Gamma_\xi$ is infinite and fixes no point of $\Lambda(\Gamma)\setminus\{\xi\}$. A parabolic point $\xi$ is called \emph{bounded parabolic} if its stabilizer $\Gamma_\xi$ acts cocompactly on \[ \Lambda(\Gamma)\setminus\{\xi\}. \]
We denote the set of bounded parabolic points by $\Lambda_{bp}(\Gamma)$.

Following Bowditch \cite{Bow95} and Roblin \cite{Roblin}, we recall the dynamical notion of geometric finiteness for discrete groups acting on proper $\CAT(-1)$ spaces.

\begin{definition}\label{def:GF}
    Let $\Ga<\Isom(X)$ be a discrete isometry group of a proper $\CAT(-1)$ space $X$. $\Ga$ is \emph{geometrically finite} if every point of the limit set is either conical or bounded parabolic; that is, \[ \Lambda(\Gamma) = \Lambda_c(\Gamma)\cup\Lambda_{bp}(\Gamma). \]
\end{definition}

\subsection{Hausdorff dimension.}
We briefly recall the definitions of Hausdorff measure and Hausdorff
dimension. Let $(X,d)$ be a metric space and let $E\subset X$. For
$\alpha>0$ and $\delta>0$, the $\alpha$-dimensional $\delta$-Hausdorff content
of $E$ is defined by
\[
\mathcal H^\alpha_\delta(E)
:=
\inf\left\{
\sum_{j\in J} \bigl(\diam U_j\bigr)^\alpha
:
E\subseteq \bigcup_{j\in J} U_j,\ 
\diam U_j\leq \delta
\right\},
\]
where the infimum is taken over all countable covers
$\{U_j\}_{j\in J}$ of $E$.

Since $\mathcal H^{\alpha}_\delta(E)$ does not decrease as $\delta\to 0$, the
limit
\[
\mathcal H^{\alpha}(E)
:=
\lim_{\delta\to 0}\mathcal H^{\alpha}_\delta(E)
=
\sup_{\delta>0}\mathcal H^{\alpha}_\delta(E)
\]
exists in $[0,\infty]$. It is called the $\alpha$-dimensional Hausdorff measure of $E$.

The Hausdorff dimension of $E$ is defined by
\[
\dim_{\mathcal H}(E)
:=
\inf\left\{
\alpha>0:\mathcal H^{\alpha}(E)=0
\right\}.
\]
Equivalently,
\[
\dim_{\mathcal H}(E)
=
\sup\left\{
\alpha>0:\mathcal H^\alpha(E)=\infty
\right\}.
\]
In particular, if
\[
\mathcal H^\alpha(E)<\infty,
\]
then
\[
\dim_{\mathcal H}(E)\leq \alpha.
\]
In particular, the same conclusion holds whenever
$\mathcal H^\alpha(E)=0$.

\section{Proof of Theorem \ref{thm:main}}\label{sec:proof}
\subsection{The projection map $\xi$}
Since $\Ga$ is non-elementary, we have $|\Lambda(\Ga)|\geq 3$. Choose a pair of distinct points $A, B\in \Lambda(\Gamma)$. Since $X$ is $\CAT(-1)$, $A$ and $B$ are joined by a bi-infinite
geodesic $c$. We then choose our basepoint $o\in c$. Let $D>0$ be a universal constant such that $\HH^n$ is $D$-hyperbolic; thus by comparison geometry, all geodesic triangles on $X$ are $D$-thin.

We now define a ``projection map'' $\xi:\Gamma\to \Lambda(\Gamma)$ as follows: for each $\gamma\in \Gamma$, since the point $\gamma o$ lies on the bi-infinite geodesic connecting $\gamma A$ and $\gamma B$. By the $D$-thin property of the geodesic triangle $\Delta(o, \gamma A, \gamma B)$, there exists $\xi(\gamma)\in \{\gamma A,\gamma B\}$ such that $\gamma o$ is $D$-close to the geodesic ray $c_{\xi(\gamma)}$. For simplicity, we also denote $\xi(\gamma)$ by $\xi_{\gamma}$. Geometrically, $\xi_\gamma$ lies in the shadow of the ball $B(\gamma o, D)$ from $o$. 

\begin{lemma}\label{lem:shadow-conv}
    If $\gamma_i o$ is a sequence of orbits converging to $\lambda\in \Lambda(\Ga)$, then $\xi_{\gamma_i}\to\lambda$.
\end{lemma}
\begin{proof}
    It is equivalent to show $(\xi_{\gamma_i},\lambda)_o\to\infty$. For the convenience, we denote $\xi_i$ short for $\xi_{\gamma_i}$. From the construction of $\xi_\gamma$, there exists $p_i\in c_{\xi_i}$ such that $d(\gamma_i o,p_i)\leq D$. Then by Lemma \ref{lem:gromov-product-lipschitz}, we have
\begin{equation}\label{eq:1}
    (p_i,\lambda)_o\geq (\gamma_i o,\lambda)_o-D.
\end{equation}
Apply Lemma \ref{lem:gromov-prod}, we have
\begin{equation}\label{eq:2}
    (\xi_{i},\lambda)_o
  \geq (p_i,\lambda)_o.
\end{equation}
Since $\gamma_i o\to\lambda$, we have $(\gamma_i o, \lambda)_o\to \infty$, so combining with \eqref{eq:1} and \eqref{eq:2}, we obtain $(\xi_i,\lambda)_o\to \infty$, or that $\xi_i(=\xi_{\gamma_i})\to \lambda$.
\end{proof}

\subsection{Spanning tree and finite $\alpha$-energy}\label{sec:AI}
We choose an arbitrary finite symmetric generating set $S$ of $\Gamma$, and choose any maximal spanning tree $\mathcal T$ of the Cayley graph $\Cay(\Gamma, S)$. Then $\mathcal T$ is a finite valence tree whose vertices are in one-to-one correspondence with $\Gamma$. The tree is rooted at the identity and the degree at each vertex is at most $|S|$. We declare the tree has edge length $1$ and for each $\gamma\in \Gamma$ we call $d_\mathcal T(\gamma,1)$ the depth of $\gamma$, denoted by $|\gamma|$. Note that, for each vertex $\gamma$, all possible vertices adjacent to $\gamma$ are of the form $\gamma s$ where $s\in S$. For any $\gamma_1,\gamma_2\in \Ga$, we denote by $[\gamma_1,\gamma_2]$ the unique geodesic path in $\mathcal T$ joining $\gamma_1, \gamma_2$.

\begin{definition}\label{def:energy}
    For a given real number $\alpha\in (0,1]$, and a given edge $e=[\gamma,\gamma s]$ on $\mathcal T$, the $\alpha$-density of the edge is defined as
    \[\ell_\alpha(e)=\rho_o(\xi_\gamma,\xi_{\gamma s})^\alpha.\]
    We call $E_\alpha(\mathcal T)$ the $\alpha$-energy of $\mathcal T$ if
    \[E_\alpha (\mathcal T)=\sum_{e\in E(T)}\ell_\alpha(e).\]
    Moreover, if $\alpha=1$, we omit the subscripts by simply writing $\ell(e)$ and $E(\mathcal T)$.
\end{definition}

The following is a key lemma for the energy estimates.

\begin{lemma}\label{lem:decay}
For every $\gamma\in\Gamma$ and $s\in S$,
\begin{equation*}
  \rho_o(\xi_\gamma,\xi_{\gamma s})
  \leq e^{L+2D}e^{-d(o,\gamma o)},
\end{equation*}
where $L=\max\{d(so,o):s\in S\}$.
\end{lemma}

\begin{proof}
By the definition of $\xi$, there exist $p\in c_{\xi_\gamma}$ and $q\in c_{\xi_{\gamma s}}$ such that $d(\gamma o,p)\leq D$ and $d(\gamma s o,q)\leq D$. By the triangle inequality, we have 
\[d(o,p)\geq d(o,\gamma o)-d(\gamma o,p)\geq d(o,\gamma o)-D,\]
and similarly,
\[d(o,q)\geq d(o,\gamma o)-d(\gamma  o,\gamma s o)-d(\gamma s o, q)\geq d(o,\gamma  o)-L-D,\]
Moreover,
\begin{equation*}
  d(p,q)\leq d(p,\gamma o)+d(\gamma o,\gamma s o)+d(\gamma s o,q)\leq L+2D.
\end{equation*}
Hence it follows from Lemma \ref{lem:gromov-prod} that
\begin{align*}
  (\xi_\gamma,\xi_{\gamma s})_o&\geq \frac{d(o,p)+d(o,q)-d(p,q)}2 \\
  &\geq \frac{1}{2}\Big((d(o,\gamma o)-D)+(d(o,\gamma o)-L-D)-(L+2D)\Big)\\
  &=d(o,\gamma o)-L-2D.
\end{align*}
Therefore, we have
\[\rho_o(\xi_\gamma, \xi_{\gamma s})=e^{-(\xi_\gamma,\xi_{\gamma s})_o}\leq e^{L+2D}e^{-d(o,\gamma o)}.\]
\end{proof}

\begin{proposition}\label{prop:energy}
    Let $\alpha\in (0,1]$ be a constant. Suppose $P_{\Gamma}(\alpha,o)<\infty$, then the $\alpha$-energy of $\mathcal T$ is finite.
\end{proposition}

\begin{proof}
    By the assumption, the Poincar\'e series
    \[P_{\Gamma}(\alpha,o)=\sum_{\gamma\in \Gamma}e^{-\alpha d(o,\gamma o)}\]
    converges. Thus, the $\alpha$-energy of $\mathcal T$ satisfies
    \begin{align*}
        E_\alpha(\mathcal T)&=\sum_{e=[\gamma,\gamma s]\in E(T)}\rho_o(\xi_\gamma,\xi_{\gamma s})^\alpha\\
        &\leq |S|\cdot \sum_{\gamma\in \Gamma}e^{\alpha(L+2D)}e^{-\alpha d(o,\gamma o)}\\
        &<+\infty,
    \end{align*}
    where the second inequality uses Lemma \ref{lem:decay} together with the fact that the degree at each vertex of $\mathcal T$ is at most $|S|$.
\end{proof}

\subsection{The boundary map}
Since $\mathcal T$ is a tree, the visual boundary of $\mathcal T$ consists of geodesic rays at $1$, thus for each $\omega\in \partial_\infty \mathcal T$, there is a unique representing sequence $\omega=(1=\gamma_0,\gamma_1,\gamma_2,\cdots)$ where $\gamma_n\in \Gamma$ is a depth $n$ element.

\begin{proposition}
     If $P_\Gamma(1,o)<\infty$, then for every $\omega=(1=\gamma_0,\gamma_1,\gamma_2,\cdots)\in \partial_\infty \mathcal T$, the corresponding sequence $(\xi_{\gamma_0},\xi_{\gamma_1},\xi_{\gamma_2},\cdots)$ is a Cauchy sequence under the visual metric. We call its limit $\Phi(\omega)$. This gives a well-defined boundary map $$\Phi:\partial_\infty \mathcal T\to \Lambda(\Gamma)\subset \partial_\infty X$$
\end{proposition}
    
\begin{proof}
    Since the positive Poincar\'e series $\sum_{\gamma\in \Gamma}e^{-d(o,\gamma o)}$ converges, for any $\epsilon>0$, there exists $N\in \N$ such that 
    \[\sum_{|\gamma|\geq N}e^{-d(o,\gamma o)}<\epsilon.\]
    Then for any $i>j\geq N$, we have by triangle inequality that
    \[\rho_o(\xi_{\gamma_i},\xi_{\gamma_j})\leq \sum_{k=i}^{j-1}\rho_o(\xi_{\gamma_k},\xi_{\gamma_{k+1}}).\]
    Apply Lemma \ref{lem:decay}, we have
    \[\rho_o(\xi_{\gamma_i},\xi_{\gamma_j})\leq e^{L+2D}\sum_{k=i}^{j-1}e^{-d(o,\gamma_k o)}\leq e^{L+2D}\sum_{|\gamma|\geq N}e^{-d(o,\gamma o)}<+\infty.\]
    It follows that $\{\xi_{\gamma_n}\}$ is a Cauchy sequence, hence by the compactness of $\Lambda(\Gamma)$, it converges to some $\Phi(\omega)\in \Lambda(\Gamma)$, so the boundary map $\Phi$ is well-defined.
\end{proof}

For the remaining section, we always assume $$P_\Gamma(1,o)=\sum_{\gamma\in \Gamma}e^{-d(o,\gamma o)}<\infty$$ so that the boundary map $\Phi$ is well-defined.

\begin{proposition}\label{prop:surjective}
If $P_\Gamma(1,o)<\infty$, then the boundary map $\Phi:\partial_\infty \mathcal T\to\Lambda(\Gamma)$ is surjective.
\end{proposition}

\begin{proof}
For any $\lambda\in\Lambda(\Gamma)$, we choose distinct $\gamma_i\in\Gamma$ such that $\gamma_i o\to\lambda$ in $\overline X=X\cup\partial_\infty X$. The convergence of $P_{\Gamma}(1,o)$ implies that only finitely many group
elements satisfy $d(o,\gamma o)\le R$ for each fixed $R$. Hence $d(o,\gamma_i o)\to\infty$.

Since $\mathcal T$ is locally finite, after passing to a subsequence, the
unique geodesic paths $[1,\gamma_i]$ converge locally to an infinite
rooted geodesic ray
\[
\omega=(1=v_0,v_1,v_2,\ldots).
\]
More precisely, set $v_0=1$. Since $v_0$ has only finitely many
children and there are infinitely many paths $[1,\gamma_i]$, there
exists a child $v_1$ of $v_0$ such that infinitely many of the paths
$[1,\gamma_i]$ contain the edge $[v_0,v_1]$.

Inductively, suppose that $v_0,\ldots,v_n$ have been chosen so that
infinitely many of the paths $[1,\gamma_i]$ contain
\[
[v_0,v_1,\ldots,v_n].
\]
Since $v_n$ has only finitely many children, one of them, denoted by
$v_{n+1}\neq v_{n-1}$, has the property that infinitely many of these paths contain
\[
[v_0,v_1,\ldots,v_n,v_{n+1}].
\]
Continuing inductively gives the desired geodesic ray $\omega$.
After a diagonal subsequence and relabeling, we may moreover assume
that, for every $n$, $\gamma_i$ is a descendant of $v_n$ whenever
$i\geq n$.

We claim that $\lambda=\Phi(\omega)$. For each vertex $v\in \mathcal T$, we denote by $\mathcal T_v$ the descendant subtree rooted at $v$. Since $P_\Gamma(1,o)<\infty$, the total energy $E(\mathcal T)$ is finite according to Proposition \ref{prop:energy}. It follows that
\[E(\mathcal T_v)\to 0, \textrm{ as }|v|\to +\infty.\]
Thus, for any $\epsilon>0$, there exists $m\in \N$ such that
\begin{equation}\label{eq:3}
    \sum_{e\subset \mathcal T_v} \ell(e)<\epsilon,
\end{equation}
whenever $|v|\geq m$. By the definition of $\Phi(\omega)$, there exists $m'\in \N$ such that
\begin{equation}\label{eq:4}
   \rho_o(\xi_{v_i},\Phi(\omega))<\epsilon
\end{equation}
whenever $i\geq m'$. We choose $N=\max\{m, m'\}$, then for any $n\geq N$, since $\gamma_n$ is a descendant of $v_{m}$, the geodesic path $[v_{m}, \gamma_n]$ lies in $\mathcal T_{v_{m}}$, thus by \eqref{eq:3} together with the triangle inequality, we have
\[\rho_o(\xi_{\gamma_n},\xi_{v_{m}})\leq \sum_{e\subset [v_{m}, \gamma_n]}\ell (e)\leq \sum_{e\subset \mathcal T_{v_{m}}}\ell(e)<\epsilon.\]
Combining with \eqref{eq:4}, we have for all $n\geq N$ that
\[\rho_o(\xi_{\gamma_n}, \Phi(\omega))\leq \rho_o(\xi_{\gamma_n}, \xi_{v_m})+\rho_o(\xi_{v_m}, \Phi(\omega))<2\epsilon.\]
This shows 
$\lim_{n\to \infty}\xi_{\gamma_n} =\Phi(\omega)$. Hence by Lemma \ref{lem:shadow-conv} we have $\lambda=\Phi(\omega)$. This shows that $\Phi$ is surjective.
\end{proof}

\subsection{Hausdorff measures}
For each $m\in \N$, we consider the finite cover of $\partial_\infty \mathcal T$ by
\[\bigcup_{|v|=m}\partial_\infty\mathcal T_v.\]
Since $\Phi$ is surjective (by Proposition \ref{prop:surjective}), its image
\[\bigcup_{|v|=m}\Phi(\partial_\infty\mathcal T_v)\]
is a finite cover of $\Lambda(\Gamma)$. We now investigate the diameter for each $\Phi(\partial_\infty\mathcal T_v)$.
\begin{proposition}\label{prop:diam}
Suppose $P_\Gamma(1,o)<\infty$, then for any vertex $v\in \mathcal T$ we have
  \[\diam(\Phi(\partial_\infty\mathcal T_v))\leq 2E(\mathcal T_v),\]
where $\mathcal T_{v}$ is the subtree at vertex $v$.
\end{proposition}
\begin{proof}
    Let $\Phi(\omega_1),\Phi(\omega_2)$ be two arbitrary element in $\Phi(\partial_\infty\mathcal T_v)$, where $\omega_1,\omega_2\in \partial_\infty\mathcal T_v$ are expressed by
    \[\omega_1=(v=u_0,u_1,u_2,\cdots),\quad \omega_2=(v=w_0,w_1,w_2,\cdots).\]
    Since all $\{u_i\}, \{w_i\}$ are descendants of $v$, we have
    \[\rho_o(\xi_{u_i},\xi_v)\leq \sum_{e\subset [v,u_i]}\ell(e) \leq E(\mathcal T_v),\]
    and similarly,
    \[\rho_o(\xi_{w_i},\xi_v)\leq E(\mathcal T_v).\]
    Passing to the limit $i\to \infty$, we have that
    \[\rho_o(\Phi(\omega_1),\xi_v)\leq E(\mathcal T_v) \textrm{ and }\rho_o(\Phi(\omega_2),\xi_v)\leq E(\mathcal T_v).\]
    By the triangle inequality, we have
    $\rho_o(\Phi(\omega_1),\Phi(\omega_2))\leq 2E(\mathcal T_v)$. Hence the proposition follows.
\end{proof}

Now we are ready to finish our proof of Theorem \ref{thm:main}.
\subsection*{Proof of Theorem \ref{thm:main}:} Assuming $P_\Gamma(\alpha,o)<\infty$ holds for some $\alpha \in (0,1]$, then certainly $P_\Gamma(1,o)<\infty$ holds, hence all results (in Section $\ref{sec:proof}$) holds up to now. If follows from Proposition \ref{prop:diam} that for each $m>0$, the total $\alpha$-volume of the cover
\[\bigcup_{|v|=m}\Phi(\partial_\infty\mathcal T_v)\]
satisfies
\begin{equation}\label{eq:volume}
    \sum_{|v|= m}\diam(\Phi(\partial_\infty\mathcal T_v))^\alpha\leq \sum_{|v|= m} 2^\alpha E(\mathcal T_v)^\alpha\leq 2\sum_{|v|= m}  E_\alpha(\mathcal T_v),
\end{equation}
where the last inequality uses the elementary subadditivity inequality that
\[\left(\sum_{i}a_i\right)^\alpha\leq \sum_{i}a_i^\alpha\]
for any positive $a_i$ and $\alpha\in (0,1]$.

Now since $P_\Gamma(\alpha,o)<\infty$, by Proposition \ref{prop:energy}, $E_\alpha(\mathcal T)$ is finite. Thus, we have
\[\sum_{|v|=m}E_\alpha(\mathcal T_v)\to 0\]
as $m\to \infty$. Thus, \eqref{eq:volume} shows that as $m\to \infty$,
 $$\sum_{|v|= m}\diam(\Phi(\partial_\infty\mathcal T_v))^\alpha\to 0.$$

On the other hand, since
\[\sum_{|v|=m}E(\mathcal T_v)\to 0,\]
we have
\[\max_{|v|=m}\diam(\Phi(\partial_\infty\mathcal T_v))\to 0\]
using Proposition \ref{prop:energy}. This shows that $\mathcal H^\alpha(\Lambda(\Ga))=0$. Consequently, $
\dim_{\mathcal H}\Lambda(\Gamma)\leq \delta(\Gamma).$

Since $\Gamma$ is non-elementary, the generalized Bishop--Jones
theorem (see \cite[Theorem~1.2.1]{DSU17}) gives
\[
\delta(\Gamma)
=
\dim_{\mathcal H}\Lambda_c(\Gamma).
\]
Since $\Lambda_c(\Gamma)\subseteq\Lambda(\Gamma),
$
we obtain
\[
\delta(\Gamma)
=
\dim_{\mathcal H}\Lambda_c(\Gamma)
\leq
\dim_{\mathcal H}\Lambda(\Gamma).
\]
Combining the two inequalities yields
\[
\dim_{\mathcal H}\Lambda(\Gamma)=\delta(\Gamma).
\]

Note that $\delta(\Gamma)<1$, it follows that
$\mathcal H^1\bigl(\Lambda(\Gamma)\bigr)=0.
$
Therefore, any connected subset of $\Lambda(\Ga)$ has diameter zero, and so it must be a singleton set. This shows $\Lambda(\Ga)$ is totally disconnected.

Finally, since $X$ is proper, its visual boundary
$\partial_\infty X$ is compact, and hence $\Lambda(\Gamma)$ is
compact. Since $\Gamma$ is non-elementary, its limit set is perfect.
Thus $\Lambda(\Gamma)$ is a nonempty, compact, perfect, totally
disconnected metrizable space, so it must be homeomorphic to the Cantor set.

\section{Proof of Theorem \ref{thm:gf}}
\subsection{Freeness} We first show that $\Ga$ must be free using Stalling's theorem. The key property is the following splitting lemma.

\begin{lem}\label{lem:split}
    If $\Gamma<\Isom(X)$ is a finitely generated torsion-free subgroup whose limit set is a Cantor set, then $\Ga$ admits a non-trivial free splitting, i.e., there exist infinite groups $H,K$ such that $\Ga=H\ast K$.
\end{lem}
\begin{proof}
    We first construct an essential cut on $\Cay (\Ga)$. Since $\Lambda(\Ga)$ is a Cantor set, we can choose a proper clopen subset $A\subset \Lambda(\Ga)$. Then $\Lambda(\Ga)=A\cup A^c$ is a decomposition into two compact sets with positive visual distance, and we denote $\epsilon_0=\rho_o(A,A^c)>0$. Pulling back the decomposition by the projection map $\xi:\Ga\to \Lambda(\Ga)$, if we denote
    $$C_1=\{\gamma\in \Ga\;|\;\xi_\gamma\in A\}\textrm{ and } C_2=\{\gamma\in \Ga\;|\;\xi_\gamma\in A^c\},$$ then $\Gamma=C_1\cup C_2$ gives a vertex decomposition on $\Cay(\Ga)$. We claim this gives an essential cut of $\Cay(\Ga)$.

    First, we notice that both $C_1, C_2$ are infinite. Indeed, we can choose an element $\lambda\in A$, since it's a limit point, there exists $\gamma_i\in \Ga$ such that $\gamma_io\to \lambda$. By Lemma \ref{lem:shadow-conv}, we have $\xi_{\gamma_i}\to \lambda$. Since $A$ is open, for all large enough $i$, we have $\xi_{\gamma_i}\in A$, hence $C_1$ is infinite. Similarly, $C_2$ is also infinite. 

    Next, we show the total number of edges between $C_1$ and $C_2$ in $\Cay(\Ga)$ is finite. Suppose $s\in S$ is an edge connecting $\gamma$ to $\gamma s$ such that $\gamma$ and $\gamma s$ belong to the different set of $C_1$ and $C_2$. By the definition of $C_1, C_2$, we have $\rho_o(\xi_{\gamma},\xi_{\gamma s})\geq \epsilon_0$. According to Lemma \ref{lem:decay}, such $\gamma$ must satisfy
    $e^{-d(o,\gamma o)}\geq e^{L+2D}\epsilon_0$, or that
    \[d(o,\gamma o)\leq L+2D-\ln(\epsilon_0).\]
    So there is only finitely many choices of $\gamma$, hence also of the pair $(\gamma,s)$ since $\Ga$ is finitely generated.

    The existence of the essential cut on $\Cay(\Ga)$ shows that $\Ga$ has at least two ends. By Stallings' theorem \cite{Stallings68}, $\Ga$ is either an amalgamated free product over a finite group, or an HNN-extension over a finite group. Since $\Ga$ is torsion-free, these finite groups must be trivial; hence in both cases, $\Ga$ has a non-trivial free splitting.
\end{proof}

For any non-elementary subgroup $H<\Ga$, its limit set is an infinite compact totally disconnected perfect subset of $\Lambda(\Ga)$, hence also a Cantor set. So we can repeat the splitting process as described in Lemma \ref{lem:split}. Since $\Ga$ is finitely generated, it has finite free rank, so the splitting process will stop at a finite stage. This shows that there is a free splitting
\begin{equation}\label{eq:split}
    \Ga=\Ga_1\ast \Ga_2\ast \cdots\ast \Ga_k
\end{equation}
such that each $\Ga_i$ is an elementary subgroup of $\Isom(X)$.

For torsion-free groups acting on a $\CAT(-1)$ space, the elementary subgroups are either of hyperbolic or parabolic type. Hyperbolic types are isomorphic to $\mathbb Z$, while the parabolic types are a priori more complicated. However, under the assumption that the critical exponent is $<1$, one can show that the parabolic subgroups also must be isomorphic to $\mathbb Z$.

\begin{lem}
    Let $P<\Isom(X)$ be a finitely generated, torsion-free discrete parabolic subgroup with $\delta(P)<1$, then $P\cong \mathbb Z$.
\end{lem}
\begin{proof}
The proof is similar to \cite[Proposition 2.3]{LiuWang}. We denote $\xi\in \partial_\infty X$ the unique fixed point of $P$, and $\mathcal H$
be the horosphere at $\xi$ through $o$. Then $P$ perserves $\mathcal H$. Let
\[c:[0,\infty)\to X\]
be the geodesic ray from $o$ to $\xi$.

Let $S$ be a symmetric generating set of $P$, and $L=\max\{d(o,so)\;|\;s\in S\}$. For each $s\in S$, by the comparison ideal triangle $\Delta(o,so,\xi)$, we have 
\[d\left(c(t),s\cdot c(t)\right)\leq Ce^{-t},\]
for some constant $C$ depends on $L$. Then for each $n\in \N$, and any $p\in P$ with word length $\leq n$, we have by triangle inequality that
\[d(c(t),p\cdot c(t))\leq Cn e^{-t},\;\forall t\geq 0.\]
Applying the triangle inequality once more, we have for any $n\in \N$ and $t\geq 0$,
\begin{align*}
    d(o,po)&\leq d(o,c(t))+d(c(t),p\cdot c(t))+d(p\cdot c(t),po)\\
    &\leq 2t+Cne^{-t}.
\end{align*}
We set $t=\log n$, then the above estimate turns to
\[d(o,po)\leq 2\log n+C.\]
Since the action of $P$ on $X$ is free, this shows that the number of elements whose length is $\leq n$ in the Cayley graph of $P$ is controlled by
\[|B_S(n)|\leq |\{p\in P\;:\;d(o,po)\leq 2\log n+C\}|,\]
and according to the definition of $\delta(P)$, the latter has asymptotic growth
\[\lesssim e^{\delta(P)(2\log n+C)}\lesssim O(n^{2\delta(P)}).\]
Thus, using $\delta(P)<1$, we see that $P$ has polynomial growth 
of degree $<2$. Hence, by Gromov's theorem \cite{Gromov} together with the Bass-Guivarc’h formula \cite{Bass72, Guivarch73}, $P$ must be virtually $\mathbb Z$. Since it's torsion-free, it must be $\mathbb Z$.
\end{proof}

Therefore, by the above lemma, each $\Ga_i$ factor in \eqref{eq:split} must be free-cyclic, thus $\Ga$ is free.

\subsection{Geometric finiteness}
Now since $\Ga$ is a finitely generated free group, it fits into Bowditch's general framework of an almost finitely presented group acting as a minimal convergence group on a Cantor set. By \cite[Theorem 1.3]{BowditchCantor}, there is an associated Bass-Serre tree $T$ where $\Ga$ acts on, such that 
\begin{enumerate}
    \item[$\bullet$] the action is cofinite,
    \item[$\bullet$] every edge stabilizer is finite, hence trivial,
    \item[$\bullet$] there is a $\Ga$-equivariant homeomorphism from the ``boundary'' $\Delta_0T$ to $\Lambda(\Ga)$.
\end{enumerate}
Here $\Delta_0T$ consists of the visual boundary of $T$ together with all vertices of infinite degree. As explained in \cite[Section 6]{BowditchCantor}, the action of $\Ga$ on $\Delta_0T$ is geometrically finite, where the conical limit points correspond to the visual boundary of $T$, and the bounded parabolic points correspond to infinite degree vertices of $T$. Therefore, $\Ga$ is geometrically finite.

\end{document}